\documentclass[12pt]{article}
\usepackage[utf8]{inputenc}
\usepackage[T1]{fontenc}
\usepackage{textcomp}
\usepackage{amsmath,amssymb,amsthm}
\usepackage{lmodern}%
\usepackage{fourier}
\usepackage{mathrsfs}
\usepackage[a4paper]{geometry}
\usepackage{graphicx}
\usepackage{microtype}
\usepackage{mathtools}
\usepackage{hyperref}
\hypersetup{pdfstartview=XYZ}

\DeclareMathOperator{\ex}{ex}
\DeclareMathOperator{\la}{La}

\renewcommand{\le}{\leqslant}
\renewcommand{\ge}{\geqslant}

\renewcommand{\geq}{\geqslant}

\newcommand{\abs}[1]{\left\lvert#1\right\rvert}

\newcommand{\littleo}[1]{o\mathopen{}\left(#1\right)}
\newcommand{\ome}[1]{\Omega\mathopen{}\left(#1\right)}

\newcommand{\La}{{\rm La}}

\newcommand{\lanp}{\La(n,P)}
\newcommand{\B}{\mathcal{B}}

\newcommand{\intervalle}[4]{#1#2\mathpunct{}\nonscript\,,#3#4}

\newcommand{\ioo}[2]{\intervalle{\left(}{#1}{#2}{\right)}}

\newtheorem{proposition}{Proposition}[section]
\newtheorem{theorem}[proposition]{Theorem}

\begin{document}
\title{Supersaturation in the Boolean lattice}

\author{Andrew P. Dove\thanks{Department of Mathematics, University of South
Carolina, Columbia, SC, USA 29208. Email: \texttt{doveap@mailbox.sc.edu}.}
\and
Jerrold R. Griggs\thanks{Department of Mathematics, University of South
Carolina, Columbia, SC, USA 29208. Email: \texttt{griggs@math.sc.edu}.
Supported in part by a visiting professorship at Univ. Paris Diderot.}
\and
Ross J. Kang\thanks{Mathematical Institute, Utrecht University.
Email: \texttt{ross.kang@gmail.com}. This author is supported by a Veni grant from the Netherlands Organisation for Scientific Research (NWO).}
\and
Jean-S\'ebastien Sereni\thanks{CNRS (LORIA), Vand{\oe}uvre-l\`es-Nancy,
France. Email: \texttt{sereni@kam.mff.cuni.cz}. This author's work was partially
supported by the French \emph{Agence Nationale de la Recherche} under reference
\textsc{anr 10 jcjc 0204 01}.}
}

\maketitle

\begin{abstract}
We prove a ``supersaturation-type'' extension of both
Sperner's Theorem (1928) and its generalization by Erd\H{o}s (1945)
to $k$-chains.
Our result implies that a largest family whose size is $x$
more than the size of a largest $k$-chain free family and that
contains the minimum number of $k$-chains is the family formed
by taking the middle $(k-1)$ rows of the Boolean lattice and
$x$ elements from the $k$th middle row.
We prove our result using the symmetric chain decomposition
method of de Bruijn, van Ebbenhorst Tengbergen, and Kruyswijk (1951).
\end{abstract}

\section{Introduction}
A core topic of extremal graph theory is the study of
``Tur\'an-type questions'': fix a
(finite) graph $H$ and a positive integer $n$.
What is the largest number $\ex(n,H)$ of edges
in an $n$-vertex graph that contains no copy of $H$?
More than a hundred years ago,
Mantel~\cite{Man07} answered this question in the case where $H$ is
$K_3$, the triangle.
About forty years later,
Tur\'an~\cite{Tur41} generalized this to all complete graphs.
More precisely, the \emph{Tur\'an graph $T(n,r)$} is
the complete $r$-partite graph of order $n$ with parts of size
$\lfloor n/r\rfloor$ or $\lceil n/r\rceil$.
Not only did Tur\'an prove that $T(n,r)$ has the largest number of edges among
all $n$-vertex graphs with no copies of $K_{r+1}$,
that is,
$\ex(n,r)=\abs{E(T(n,r))}$,
but also he proved that
all other $n$-vertex graphs containing no copies of $K_{k+1}$
have strictly fewer edges than $T(n,r)$.

The theory of graph supersaturation deals with the situation
beyond the threshold given by $\ex(n,H)$. Specifically, define
$\ell(n,H,q)$ as the least number of copies of $H$ in an
$n$-vertex graph with at least $\ex(n,H)+q$ edges. By the
definition, we know that $\ell(n,H,q)\ge1$ as soon as $q>0$, but
it turns out that an extra edge is likely to create many more
copies of $H$. Arguably, the first result in this direction was
proved in an unpublished work of Rademacher from 1941 (orally communicated
to Erd\H{o}s~\cite{Erd55}): while
Mantel's theorem states that every $n$-vertex graph with more than
$\abs{E(T(n,2))}=\lfloor n/2\rfloor \lceil n/2\rceil$ edges
contains a triangle, Rademacher established that such graphs
contain, actually, at least $\lfloor n/2\rfloor$ triangle copies.

This result was generalized by Erd\H{o}s, who proved that
$\ell(n,K_3,q)\geq\lfloor n/2\rfloor$ first if $q\in\{1,2,3\}$ in
1955~\cite{Erd55} and a few years later in the case $q<c\cdot n/2$ for a fixed constant
$c\in\ioo{0}{1}$~\cite{Erd62b}. More than twenty years
later, Lov\'asz and Simonovits~\cite{LoSi83} established the
following theorem, thereby confirming a conjecture of Erd\H{o}s.
\begin{theorem}[Lov\'asz and Simonovits (1983)]
Let $n$ and $q$ be positive integers.
If $q<n/2$, then $\ell(n,K_3,q)\ge q\cdot\lfloor n/2\rfloor$.
\end{theorem}
\noindent
In addition, Lov\'asz and Simonovits~\cite{LoSi83} determined
$\ell(n,K_r,q)$ when $q=\littleo{n^2}$.
Their techniques do not apply, though, for the case where $q=\ome{n^2}$.
Solutions to this difficult problem were provided recently with the aid of flag algebras:
first, by Razborov~\cite{Raz08} for $H=K_3$,
then, by Nikiforov~\cite{Nik11} for $H=K_4$ and,
finally, by Reiher~\cite{Rei11} for the general case $H = K_r$.

Supersaturation results have not to our knowledge been studied as
extensively in other important areas of extremal combinatorics. In
this paper, we pursue this direction for extremal set
theory.

Let the \emph{Boolean lattice $\mathcal{B}_n$} be the poset
$(2^{[n]},\subseteq)$ of all subsets of the set $[n] =
\{1,\dots,n\}$, ordered by inclusion. For a set $S$, the
collection of all $k$-subsets of $S$ is denoted by $\binom{S}{k}$.
Following notation of previous work~\cite{GLL12}, by $\B(n,k)$ and $\Sigma(n,k)$ we mean the
families of subsets of $[n]$ of the $k$ middle sizes and the size
of the families. More precisely,
\begin{align*}
\B(n,k)=\binom{[n]}{\left\lfloor
\frac{n-k+1}2 \right\rfloor}\cup\cdots \cup\binom{[n]}{\left\lfloor \frac{n+k-1}2 \right\rfloor}
\ \ \text{ or } \ \
\B(n,k)= \binom{[n]}{\left\lceil \frac{n-k+1}2 \right\rceil}\cup\cdots\cup
\binom{[n]}{\left\lceil \frac{n+k-1}2 \right\rceil}
\end{align*}
(so, depending on the parity
of $n$ and $k$, this can be either one or two different families).

Given two finite posets $P = (P,\le)$ and $P' = (P',\le')$, we say
that \emph{$P'$ contains $P$} (or \emph{$P$ is a  (weak) subposet of
$P'$}), if there is an injection $f: P \to P'$ that preserves the
partial ordering, i.e.~if $u \le v$ in $P$, then $f(u) \le' f(v)$
in $P'$.
We let $\mathcal{P}_k$ be the $k$-element totally ordered
poset (chain).

What is the largest size $\lanp$ of a family of subsets of $[n]$
that does not contain $P$? The foundational result of this kind,
Sperner's Theorem~\cite{Spe28} from 1928, answers this question
for a two-element chain: $\la(n,\mathcal{P}_2) = \Sigma(n,1)$. Moreover,
the value $\la(n,\mathcal{P}_2)$ is attained only by $\B(n,1)$, which consists
of subsets all of (a) middle size. Erd\H{o}s~\cite{Erd45} almost two
decades later generalized this to $P = \mathcal{P}_k$, showing
that $\la(n,\mathcal{P}_k) = \Sigma(n,k-1)$, which is attained only by  $\mathcal{B}(n,k-1)$. Katona has championed the $\la(n,P)$
problem for posets $P$ other than $k$-chains, and this is
a challenging area of extremal set theory.  It is often very
difficult to obtain the extremal size $\lanp$ of such a family,
even asymptotically (see~\cite{GLL12} for a survey).

However, for those $P$ for which we know the exact threshold, we
can ask how many copies of $P$ must be present in families larger
than the threshold $\lanp$. Here we investigate the simplest
instance of this problem, when $P$ is a chain.  Analogous to the
way that Rademacher and Erd\H{o}s (and subsequent researchers)
have extended the theorems of Mantel and Tur\'an, we present a
supersaturation extension of Sperner's Theorem and its $k$-chain
generalization by Erd\H{o}s.

Our initial result was a lower bound on the number of
$\mathcal{P}_2$'s in a family $F\subseteq2^{[n]}$ of a given
size that is optimal for $\abs{F}\le \Sigma(n,2)$, extending Sperner's Theorem.  By investigating
more examples, we came to believe that for any size $\abs{F}$, with
$\Sigma(n,\ell)\le \abs{F}\le \Sigma(n,\ell+1)$, the number of
$\mathcal{P}_2$'s in $F$ is minimized by taking $F$ to consist of
$\B(n,\ell)$ together with subsets of $\B(n,\ell+1)$.
In further exploration of problems related to poset-free families of subsets, we came across the work of Kleitman~\cite{Kl68} from 1968, which corroborates our findings and intuition.
Indeed, Kleitman, albeit with matching theory techniques, obtained the (same) supersaturation extension of Sperner's Theorem and more.  This settled a conjecture of Erd\H{o}s and Katona.  In particular, he determined the minimum number of pairs
$(A,B)$ with $A\subset B$ in a family $F\subseteq 2^{[n]}$ of \emph{any}
given size.  As we had intuited, taking the subsets of some middle
sizes attains the optimum.

One particularly nice way to quickly derive Sperner's Theorem and
its generalization by Erd\H os  is to employ the remarkable
symmetric chain decomposition (\emph{SCD}, for short) of all $2^n$
subsets of $[n]$, discovered by de Bruijn, van Ebbenhorst
Tengbergen, and Kruyswijk~\cite{BEK51} in 1951.  It is a partition
of the Boolean lattice into just $\binom{n}{\lfloor n/2 \rfloor}$ disjoint chains of
subsets, where for each chain there is some $k\le n/2$ such
that the chain consists of a subset of each size from $k$ to
$n-k$.  For all $k$ the decomposition induces the best possible
upper bound on $\abs{F}$ for a $\mathcal{P}_k$-free family $F$ of
subsets of $[n]$.  (It requires some
additional arguments to obtain the extremal families.) The construction, which is obtained by a
clever inductive argument, was done originally in the more general
setting of a product of chains.  In this way, the authors obtained
the extension of Sperner's Theorem to the lattice of divisors of
an integer $N$.

There is a large literature on the existence of SCDs in posets and
other ordered/ranked set systems~\cite{GrKl78,Gri77,LDD94}. Greene
and Kleitman~\cite{GrKl76} discovered an explicit SCD of the
Boolean lattice for all $n$, based on a simple ``bracketing
procedure", as opposed to the original inductive construction.
Bracketing has proven to be valuable in its own right, such as for
the Littlewood-Offord problem~\cite{GrKl78} and for the
construction of symmetric Venn diagrams on~$n$ sets for all
prime~$n$~\cite{GKS04}.

It is not surprising then that a SCD of $\mathcal B_n$
yields a lower bound on the number of paths in a family $F$ of
given size.  In particular, if we arbitrarily consider one particular SCD, the number
of chains in $F$ that are also chains in the
SCD is minimized by taking the sets of $F$ to be of the
middle sizes. However, this argument does not account for the many containment
relations for pairs of subsets $A\subset B$ where $A$ and $B$ are on
different chains in the SCD.  To adjust for this, and to
exploit symmetry by avoiding bias towards a particular SCD, our
new idea here is to take all $n!$ SCDs obtained by permutation of the ground set $[n]$.  In this way, we obtain
lower bounds on the number of paths in a family $F$ of given size,
bounds that are best possible for small $F$.

Our main aim in this paper is then to prove the following
supersaturation extension of the theorems of Sperner and
Erd\H{o}s, using the above-outlined SCD approach.
\begin{theorem}\label{thm:supersaturation}
If a family $F$ of subsets of $[n]$ satisfies
$\abs F = x + \Sigma(n,k-1)$,
then there must be at least
\[
x\cdot\prod_{i=1}^{k-1}\left(\left\lfloor\frac{n+k}{2}\right\rfloor-i+1\right)
\]
copies of $\mathcal{P}_k$ in $F$.
\end{theorem}
\noindent
Note that
\[
\prod_{i=1}^{k-1}\left(\left\lfloor\frac{n+k}{2}\right\rfloor-i+1\right)
\]
is the number of copies of $\mathcal P_k$ contained in $\mathcal B(n,k)$, with one endpoint of the chain being a particular set in the $k$th middle row.
Thus the family that consists of $\mathcal B(n,k-1)$ and $x$ sets from the $k$th middle row witnesses that the above bound is tight for
\[
x \le \binom{n}{\left\lfloor \frac{n}2\right\rfloor + (-1)^k\left\lfloor \frac{k}2\right\rfloor}.
\]

More generally, Kleitman~\cite{Kl68} has conjectured that for any $k$ the natural
construction (that selects subsets around the middle) minimizes
the number of chains $\mathcal{P}_k$ in $F$.  Our result
gives new information in support of this conjecture, verifying it
for $\abs{F}\le\Sigma(n,k)$.
We suspect that a stronger version of our SCD method, in which
weights are assigned, may lead to a proof of Kleitman's conjecture in full for general $k$.  So far our efforts in this direction that looked
very promising have not yet succeeded.  We cannot imagine that his
conjecture is not correct.

\section{Proof of Theorem~\ref{thm:supersaturation}}
As mentioned earlier, we shall use the symmetric chain decomposition of $\mathcal B_n$.

\begin{proof}[Proof of Theorem~\ref{thm:supersaturation}]
Given a poset $(P,\preceq)$ on $2^{[n]}$, let us say that a $k$-chain $A_1 \subset \dots \subset A_k$ of $F$ (in $\mathcal{B}_n$) is \emph{included} in $P$ if $A_1 \prec \dots \prec A_k$, and furthermore define $c_F(P)$ to be the number of $k$-chains of $F$ included in $P$.  For any SCD $\mathcal{C}$ of $\mathcal{B}_n$, let $P_\mathcal{C}$ be the poset on $2^{[n]}$ defined by taking the disjoint union of the chains in $\mathcal{C}$.
Let us fix the SCD $\mathcal{C}$.
By the pigeonhole principle, $P_\mathcal{C}$ includes at least $x$ $k$-chains of $F$, i.e.~$c_F(P_\mathcal{C}) \ge x$.
Each (non-trivial) permutation $\pi$ of $[n]$ applied to
$\mathcal{C}$ results in a new unique SCD $\pi(\mathcal{C})$ for
$\mathcal{B}_n$. Note that $\pi(\mathcal{C})\neq\pi'(\mathcal{C})$ for
distinct permutations $\pi$ and $\pi'$ of $[n]$.
By summing over the permutations $\pi$ of $[n]$, we obtain
\[
n!\cdot x \le \sum_\pi c_F(P_{\pi(\mathcal{C})}).
\]
Let us change the summation to sum over all $k$-chains of $F$.  For this, we define $N(n,A_1,\dots,A_k)$ to be the number of permutations $\pi$ such that $P_{\pi(\mathcal{C})}$ includes a given chain $A_1 \subset \dots \subset A_k$ of $F$.  We obtain
\[
\sum_\pi c_F(P_{\pi(\mathcal{C})}) = \sum_{A_1 \subset \dots \subset A_k; A_i \in F} N(n,A_1,\dots,A_k).
\]
Setting $a_i\coloneqq\abs{A_i}$ for each $i\in\{1,\ldots,k\}$,
it holds that
\[
N(n,A_1,\dots,A_k) = a_1!\cdot(a_2-a_1)!\cdots(a_k-a_{k-1})!\cdot(n-a_k)!\cdot\min\left\{\binom{n}{a_1},\binom{n}{a_k}\right\},
\]
where the last factor comes from the number of chains in a SCD that the given
chain could fit.  After some manipulation, we deduce that
\[
N(n,A_1,\dots,A_k) = \frac{n!}{\max\left\{
\binom{a_k}{a_{k-1}}\cdots\binom{a_2}{a_1},
\binom{n-a_1}{n-a_2}\cdots\binom{n-a_{k-1}}{n-a_k}
\right\}}.
\]
We shall find a general upper bound for $N(n,A_1,\dots,A_k)$ by minimizing the maximum of $y$
defined as $\binom{a_k}{a_{k-1}}\cdots\binom{a_2}{a_1}$ and $z$ defined as $\binom{n-a_1}{n-a_2}\cdots\binom{n-a_{k-1}}{n-a_k}$.  Note the following binomial identity:
\[
\binom{a+i+j}{a+i}\binom{a+i}{a} = \binom{a+i+j}{a+j}\binom{a+j}{a}.
\]
As a consequence of this, the values of $y$ and $z$ are invariant as long as
the multiset of all differences between consecutive values of $a_i$ is
invariant.  By this fact, if there is some difference in this multiset that is
at least $2$, we may assume without loss of generality that this ``large''
difference is between $a_{k-1}$ and $a_k$. It follows that
\[
y'\coloneqq y\cdot\frac{\binom{a_{k-1}+1}{a_{k-1}}}{\binom{a_k}{a_{k-1}}}
= y\cdot\frac{a_{k-1}+1}{\binom{a_k}{a_{k-1}}}
< y,
\]
provided that $a_{k-1} > 0$. Similarly,
\[
z'\coloneqq z\cdot\frac{\binom{n-a_2+1}{n-a_2}}{\binom{n-a_1}{n-a_2}}
= z\cdot\frac{n-a_2+1}{\binom{n-a_1}{n-a_2}}
< z,
\]
provided that $a_2 < n$.  It follows that $y$ and $z$ are minimized when the multiset of differences is the multiset of all ones, i.e.~with
\[
y = \frac{a_k!}{(a_k-k+1)!}
\quad\text{and}\quad
z = \frac{(n-a_k+k-1)!}{(n-a_k)!}.
\]
The maximum of $y$ and $z$ is then minimized by choosing $a_k$ to be $\left\lfloor \frac{n+k}{2}\right\rfloor$, so
\[
n!\cdot x \le
\sum_{\substack{A_1 \subset \dots \subset A_k\\A_i \in F}} N(n,A_1,\dots,A_k)
 \le \sum_{\substack{A_1 \subset \dots \subset A_k\\A_i \in F}} \frac{n!}{\prod_{i=1}^{k-1}\left(\left\lfloor\frac{n+k}{2}\right\rfloor-i+1\right)}
= c_F(F)\cdot \frac{n!}{\prod_{i=1}^{k-1}\left(\left\lfloor\frac{n+k}{2}\right\rfloor-i+1\right)},
\]
as required.
\end{proof}

\section*{Remarks}
Our result was presented by the second author in Prague in June
2012~\cite{Gri12}. In the preparation of this manuscript, we learned that recently Das, Gan and
Sudakov have independently pursued a similar line of research and possibly obtained results similar to ours.

\end{document}